\documentclass{amsart}

\usepackage{amssymb,amsmath,enumerate,color}
\usepackage{amsfonts, booktabs}
\usepackage{latexsym,parskip}
\usepackage{hyperref}
\hypersetup{pdfauthor=MS} \hypersetup{pdftitle=Fields}

\definecolor{red}{rgb}{.8,0,0}

\newcommand{\PP}{{\bf P}}

\newcommand{\Q}{{\bf Q}}
\newcommand{\QQ}{{\mathbb Q}}

\newcommand{\Z}{{\bf Z}}

\newcommand{\OO}{\mathcal{O}}

\def\ip(#1,#2){{\langle#1,#2\rangle}}

\newcommand{\Gal}{\mathop{\rm Gal}\nolimits}

\newcommand{\NS}{\mathop{\rm NS}\nolimits}

\newcommand{\XX}{{\mathcal X}}

\def\NeS{N\'eron--Severi}
\def\MoW{Mordell--\mbox{\kern-.12em}Weil}
\def\MOW{MORDELL--\mbox{\kern-.12em}WEIL}

\def\0{^{\phantom{0}}}
\def\9{_{\phantom{0}}}

\theoremstyle{plain}
\newtheorem{theorem}{Theorem}
\newtheorem{prop}[theorem]{Proposition}
\newtheorem{lemma}[theorem]{Lemma}

\def\be{\begin{equation}}
\def\ee{\end{equation}}
\def\benn{\begin{equation*}}
\def\eenn{\end{equation*}}
\def\bea{\begin{eqnarray}}
\def\eea{\end{eqnarray}}

\def\MoW{Mordell--\kern-.1emWeil}

\begin{document}

\title[An elliptic K3 surface $\XX/\QQ(t)$ with \MOW\ rank~$17$, 
 I: formulas]{\hbox{An elliptic K3 surface $\XX/\QQ(t)$ with \MOW\ rank~$17$}
 \hbox{I: Formulas for $\XX$ and base changes of ranks $18$ and $19$}
}

\author{Noam D.~Elkies}

\date{25 August 2026}

\begin{abstract}
  In \cite{NDE:NMBRTHRY28,NDE:MFO}
  we announced a K3 surface $\XX/\QQ$ with \NeS\ group
  $\NS(\XX) = \NS_\Q(\XX)$ of rank~$19$ and an elliptic fibration
  of \MoW\ rank $17$.  This is the largest possible \MoW\ rank over~$\Q(t)$
  for an elliptic K3 surface.  One of the fibers of this surface is
  the elliptic curve of rank at least~$28$ that was the elliptic curve
  of highest rank known during the years 2006--2024.
  We further announced that there are quadratic base changes
  of this fibration to elliptic surfaces of \MoW\ rank~$18$ over~$\Q(t)$,
  and that there are pairs of such quadratic base changes whose compositum is
  an elliptic surface of \MoW\ rank~$19$ over an elliptic curve $E_0/\Q$ with
  $\#(E_0(\Q)) = \infty$, whence there are infinitely many
  elliptic curves of rank at least~$19$ over~$\Q$.  We exhibit one such pair.
  A subsequent paper will show (using a techniques also announced
  in~\cite{NDE:MFO} how we computed $\XX$\/ and it fibration.
\end{abstract}

\maketitle


\section{Introduction}
\label{s:intro}

In the post \cite{NDE:NMBRTHRY28} to the NMBRTHRY mailing list,
we announced several records for ranks of elliptic curves,
and later gave further information in the notes \cite{NDE:MFO}
for a lecture series at an Oberwolfach meeting on computational number theory.
Here we finally give proofs of these results:

\begin{theorem}
\label{thm:rk17_0}
There is a K3 surface $\XX/\Q$ and an elliptic fibration
$t : \XX \to \PP^1$ whose generic fiber has rank $17$ over~$\Q(t)$,
the largest possible for a K3 surface over~$\Q$.
\end{theorem}

\begin{theorem}
\label{thm:rk28}
The fibration $\XX$ has a fiber that is an elliptic curve of rank
at least~$28$ over~$\Q$.
\end{theorem}

\begin{theorem}
\label{thm:rk18,19}
a) The fibration $t : \XX \to \PP^1$ has quadratic base changes
$t : \PP^1 \to \PP^1$ to elliptic surfaces of rank at least~$18$
over $\PP^1_{\!\Q}$.
\\
b) There are pairs of fibrations from~(a) whose compositum gives
biquadratic base changes $E_0 : \PP^1 \to \PP^1$ to a elliptic surfaces
of rank at least~$19$ over some elliptic curve $E_0 / \Q$ of positive rank.
\\
c) There are infinitely many elliptic curves $E / \Q$ of rank at least~$19$.
\end{theorem}

Theorem~\ref{thm:rk17_0} was announced some months earlier,
before we had obtained explicit equations for $\XX$ and~$t$,
because it requires only a non-CM point on a Shimura curve
(which arises as a moduli space for suitably polarized K3 surfaces).
This also required some computation, but was much easier because
the Shimura curve happens to have genus~$2$ with bielliptic involutions
whose quotients were already tabulated.  The curve and its CM and known
non-CM points are exhibited in~\cite{NDE:MFO}.

For Theorem~\ref{thm:rk18,19},
part~(a) can also be proved without an explicit model for $(\XX,t)$.
Then parts (b) and~(c) nearly follow: there are many choices
of the two quadratic base changes that guarantee $E$\/ has rational points,
and it seems most unlikely that all such points are torsion;
but we could not prove this without computing $\XX$ and~$t$.
Part~(c) then follows from~(b) by Silverman's specialization theorem.
As far as we know $18$ and $19$ remain the records for
the rank of a nonconstant elliptic curve over~$\Q(t)$
and the $\limsup$ of the ranks of elliptic curves over~$\Q$.

We did need an explicit formula for $(\XX,t)$
also to start searching for fibers of record rank
such as the one that proved Theorem~\ref{thm:rk28}.
In~\cite{NDE:NMBRTHRY28} we exhibited a fiber with $28$ independent points, and
in~\cite{NDE:MFO} we explained how we searched for candidate high-rank fibers
and for extra generators on such fibers.  As far as we know,
before the AI explosion of the summer of 2026
every known elliptic curve of rank $ > 24 $ over~$\Q$
was obtained from a high-rank elliptic fibration of the same K3 surface~$\XX$.

We did not exhibit our Weierstrass equation for $(\XX,t)$,
whose coefficients have a total of about $800$ digits,
plus another $1200+$ digits for  the \hbox{$x$-coordinates} $x_i$ of
$P_i = (x_1(t),y_1(t))$ through $(x_{17}(t),y_{17}(t))$
(the $y_i$ are then determined up to sign by the Weierstrass equation).
We expected that reconstructing an equation from
the information in~\cite{NDE:MFO} would be an exercise,
albeit a difficult one; the published \hbox{rank-$28$} fiber
(and the fibers of rank $25,26,27$ published later)
probably would not help, but could be used to check that one has found
the intended solution, by using the \hbox{$j$-invariants} to solve for~$t$.

To our knowledge --- and quite to our surprise ---
this puzzle remained unsolved until earlier this month when
Michael Rubinstein \cite{Rubinstein} e-mailed us an equation,
reporting that he used both Claude and Grok to do this,
and indeed locating those curves of rank~$28$ and~$27$ as fibers.
Coincidentally this happened a day before the first published
\hbox{rank-$30$} curve appeared at \cite[Curve~273]{ICARM},
found by Alp\"{o}ge and Howell.  They also used Claude;
as of this writing, they have not revealed their method,
but we checked that their curves of rank~$30$ (and now~$31$
\cite[Curve~302]{ICARM}) are \underline{not} in our fibration. 

In the next section of this paper we exhibit equations for
$(\XX,t)$ and~$x_i(t)$ (and the first \hbox{$y$-coordinate} $y_1(t)$).
We then use the height pairing to check that
the $P_i$ are independent in the elliptic-curve group law.
The maximum rank of an elliptic K3 surface in characteristic zero
is~$18$, but it was already known that this maximum cannot be attained
over~$\Q$.  Thus we prove Theorem~\ref{thm:rk17_0} with an explicit
fibration and generators (Theorem~\ref{thm:rk17}).  We then exhibit
the \hbox{$t$-coordinate} of our fiber of rank at least~$28$, thus proving
Theorem~\ref{thm:rk28}, and also the \hbox{$t$-coordinates} of
fibers of rank at least $25$, $26$, and~$27$ that we announced later.

In the final section we prove Theorem \ref{thm:rk18,19}
via a characterization (Proposition~\ref{prop:6}) of
quadratic base changes coming from rational quadratic sections,
and exhibit one choice of $E_0$.

We give long polynomials, matrices, etc.\ in machine-readable form
(specifically GP/PARI syntax, but this is easy to convert to
other packages such as Magma), so that one can easily copy-and-paste
the equations from the PDF file to a computer algebra system 
--- or these days to an AI session --- for checking and further investigation.

\section{An elliptic K3 surface of \MoW\ rank~$17$ over~$\Q$}
\label{s:X}

We prove Theorem~\ref{thm:rk17_0} by exhibiting an elliptic fibration
and the $x$-coordinates of $17$ sections, and checking that they are
independent.  Since it is already known that every elliptic K3 surface
has rank $< 18$ over~$\Q(t)$, this proves that the rank is exactly~$17$.
Finally we exhibit the \hbox{$t$-coordinates} of the four fibers
that we found in April and May~2006 with \MoW\ ranks are at least
$25, 26, 27, 28$.

\subsection{The elliptic fibration.}
\label{ss:X}
We shall prove:

\begin{theorem}
\label{thm:rk17}
Let $\XX/\Q(t)$ be the elliptic K3 surface
\be
\label{eq:X}
\XX: y^2 = x^3 - 27 S(t) x + \frac{27}{4} T(t)
\ee
where $S,T \in \Z[t]$ are the polynomials
{\small
\begin{verbatim}
{
S = 307516108335972163537936*t^8 + 10476571005172375234427296*t^7
   + 234256046667228607566274912*t^6 + 2020678721371875903158954848*t^5
   + 8789387383568632081365832240*t^4 + 21430123310022469548285709072*t^3
   + 28607402618712438778345257832*t^2 + 17860826619093915900857289304*t
   + 4201305425690184127251888481;

T = 1050290276365892761266194577222156800*t^12
   + 67802587761728815952013525763236564480*t^11
   + 2392486076703808362288120169049836903680*t^10
   + 38126035250980128714796491999580538771200*t^9
   + 372202978476351718721663756748866085220800*t^8
   + 2373760737463050257069464720014664373086080*t^7
   + 9904246958414858348647761354992989326760320*t^6
   + 26905633537996991160744810870319331164617600*t^5
   + 47243082583908684509409509915652973906060800*t^4
   + 52862444598312784274784438443066814490530880*t^3
   + 36435013603665838306995466090052055171475872*t^2
   + 13865015501478235534002649882546248548532768*t
   + 2193201312876924214657300134273061462776968
}.
\end{verbatim}
}
\end{theorem}

{\em Proof}\/: We exhibit the \hbox{$x$-coordinates} $x_1,\ldots,x_{17}$
of $17$ solutions $P_i = (x_i(t), y_i(t))$ ($1 \leq i \leq 17$)
of~(\ref{eq:X}) in polynomials of degrees $4,6$, and check that
they are independent in the \MoW\ group of~$\XX$\/ by computing
the height pairings $(P_i,P_j)$ and verifying that the Gram matrix
$\bigl((P_i,P_j)\bigr)_{i,j=1}^{17}$ has nonzero determinant.

Since any pair $(x,\pm y)$ sums to zero in the group law,
it is enough to exhibit the~$x_i$, because the choice of~$y_i$
does not affect the rank.  We must still make a choice in order to compute 
height pairings $(P_i,P_j)$.  Having exhibited~$x_i$, we can specify
the choice of $y_i$ by giving just the sign of its leading coefficient.
Nevertheless, for $i=1$ we exhibit both $x_1(t)$ and $y_1(t)$ in full
to guard against typographical or copying errors in these coordinates
as well as the coefficients $S(t)$ and $T(t)$:

\begin{verbatim}
{ x1 = 419884536396*t^4 - 6900780974412*t^3 + 84146613883956*t^2
      + 448019664127620*t + 304456582100883;
  y1 = -1917605876395727232*t^6 - 102352278854532258864*t^5
      - 1140847719698231045748*t^4 - 4035207954948742785564*t^3
      - 3519107150812739581680*t^2 + 3523393851784245137088*t
      + 2913630401455186533120 }
\end{verbatim}

The remaining $x$-coordinates $x_2,\ldots,x_{17}$ are

{\tiny
\begin{verbatim}
{[ 960407733324*t^4 + 32982109084392*t^3 + 228946163155128*t^2 + 451794461129916*t + 136981770876723,
   191092188813708*t^4 + 1333867432517100*t^3 + 1886874275645632*t^2 - 1605097821400112*t - 329325794912045,
   881331598668*t^4 + 20106018946320*t^3 + 191145949680312*t^2 + 489728900722308*t + 205182206512275,
   7087168886668*t^4 + 60084894852268*t^3 + 119727776998960*t^2 + 70169296825056*t + 124968924115923,
  -2654105330292*t^4 - 75993994150932*t^3 - 436645278959760*t^2 - 794565531069024*t - 379495133581677,
  -2842297467828*t^4 - 76265967628812*t^3 - 435348555495516*t^2 - 839785632779964*t - 362275017421677,
  -2947994548863*t^4 - 25598671575906*t^3 + 63026792718435*t^2 + 464786649518334*t + 320073994727283,
  -2414976971316*t^4 + 54450640822344*t^3 + 2834155382615496*t^2 - 12715093268802228*t + 9732633560757363,
   13412195434209*t^4 - 250429886278338*t^3 - 242466751598877*t^2 + 666777676835166*t + 521473683384723,
   353434406988*t^4 + 62514191744628*t^3 + 271192708423620*t^2 - 55497536934924*t - 125053466701677,
  -1586365228500*t^4 - 83483171473260*t^3 - 486949775116428*t^2 - 653155402766412*t + 345607319019603,
   5724934740993*t^4 - 29562185743194*t^3 - 764745203764737*t^2 - 1492345666793982*t + 2685202943830203,
   2426651051916*t^4 + 18957579322812*t^3 + 164233784236041*t^2 + 694007861500356*t + 964931580370398,
   3734241561804*t^4 + 27061905787332*t^3 + 153585168336648*t^2 + 413697107315976*t + 260813404752123,
  -798764561556*t^4 - 57890934328188*t^3 - 354534107851872*t^2 - 678508004328024*t - 287286844790877,
   2607059076492*t^4 + 25332675721548*t^3 + 206429765168484*t^2 + 494998206601236*t + 408936541867923 ]}
\end{verbatim}
}

We choose $y_1,\ldots,y_{17}$ whose leading coefficients have signs
\be
\label{eq:y_signs}
- + - + + + + + - + - + + + - + +
\ee
and calculate the resulting Gram matrix~$G$ of height pairings
(see Lemmas~\ref{lem:ht} and \ref{lem:ip} below):

\begin{verbatim}
 {[ 4,-2,-2,-2,-2,-2,-2,-2,-2,-2,-2,-2,-2,-2,-2,-2, 1;
   -2, 4, 2, 1, 0, 1, 0, 1, 1, 1, 1, 2, 1, 1, 1, 1, 1; 
   -2, 2, 4, 2, 1, 0, 1, 2, 1, 1, 1, 0, 1, 1, 2, 1, 0;
   -2, 1, 2, 4, 1, 0, 1, 1, 2, 1, 1, 1, 0, 1, 0, 1, 0; 
   -2, 0, 1, 1, 4, 2, 0, 1, 1, 0, 1, 1, 1, 2, 1, 0,-1;
   -2, 1, 0, 0, 2, 4, 1, 0, 0, 1, 2, 1, 1, 1, 1, 0,-1; 
   -2, 0, 1, 1, 0, 1, 4, 0, 0, 1, 1, 0, 1, 1, 2, 2,-2;
   -2, 1, 2, 1, 1, 0, 0, 4, 1, 2, 1, 0, 1, 1, 1, 1, 0; 
   -2, 1, 1, 2, 1, 0, 0, 1, 4, 2, 2, 1, 2, 2, 1, 1,-1;
   -2, 1, 1, 1, 0, 1, 1, 2, 2, 4, 2, 0, 2, 1, 2, 1,-1; 
   -2, 1, 1, 1, 1, 2, 1, 1, 2, 2, 4, 0, 1, 2, 2, 1,-2;
   -2, 2, 0, 1, 1, 1, 0, 0, 1, 0, 0, 4, 0, 1, 0, 1, 1;
   -2, 1, 1, 0, 1, 1, 1, 1, 2, 2, 1, 0, 4, 2, 2, 1,-2;
   -2, 1, 1, 1, 2, 1, 1, 1, 2, 1, 2, 1, 2, 4, 2, 1,-2;
   -2, 1, 2, 0, 1, 1, 2, 1, 1, 2, 2, 0, 2, 2, 4, 1,-2;
   -2, 1, 1, 1, 0, 0, 2, 1, 1, 1, 1, 1, 1, 1, 1, 4, 0;
    1, 1, 0, 0,-1,-1,-2, 0,-1,-1,-2, 1,-2,-2,-2, 0, 4 ]}
\end{verbatim}

It has nonzero determinant~$948$, so $P_1,\ldots,P_{17}$ are independent
as claimed.

\subsection{The height pairing.}
\label{ss:pair}
It remains to explain the computation of the entries $(P_i,P_j)$ of~$G$.
We regard (\ref{eq:X}) as an equation defining $\XX$\/ as a surface
of degree~$12$ in the $(1,1,4,6)$ weighted projective space 
with coordinates $t,1,x,y$.  The elliptic fibration $\XX \to \PP^1$
is the map $(t:1)$; elements of the \MoW\ group are sections of the fibration,
and a nonzero section is a solution $(x,y)$ of~(\ref{eq:X})
in rational homogeneous forms in~$(t:1)$ of degrees~$4$ and~$6$.
Here a ``rational homogeneous form'' of degree~$w$ is a quotient $N/D$\/
of homogeneous polynomials of degrees $d,d-w$ for some $d \geq w$,
with $D \neq 0$.
Since the numerator~$N$\/ and denominator~$D$\/ are binary forms,
we may assume (by clearing common factors) that $N,D$\/ are relatively prime.

\begin{lemma}
\label{lem:ht}
Let $P = (x,y)$ be any nonzero element of the \MoW\ group of~(\ref{eq:X}),
with $x = N/D$\/ for some homogeneous forms $N,D$\/ of degrees $d,d-4$
without common factor.  Then the na\"{\i}ve and canonical heights of~$P$\/
both equal~$d$.
\end{lemma}

{\em Proof}\/: Since $\XX$ has no reducible fibers,
the na\"{\i}ve and canonical heights are equal.  We can compute
the na\"{\i}ve height as an intersection number $-(s_P-s_0)^2$,
which equals $4 + 2 s_0 \cdot s_P$ because $s_0 \cdot s_0 = s_P \cdot s_P = -2$.
Since the function $x$ on~$\XX$\/ has a double pole at~$s_0$
and no other poles, $2 s_0 \cdot s_P$ is the degree of the denominator~$D$,
which equals $d-4$.  Hence the height is $4 + (d-4) = d$, as claimed. \qed

In particular the height is at least~$4$, with equality if and only if
$x,y$ are homogeneous polynomials of degrees $4,6$
(so in terms of the affine coordinate~$t$ they are polynomials
of degrees $\leq 4$ and $\leq 6$).  We call such $(x,y)$
an {\em integral point}, and $s_P$ an {\em integral section}
of the fibration; geometrically, $s_P$ is integral if and only if
it is disjoint from~$s_0$.

\begin{lemma}
\label{lem:ip}
Let $P = (x,y)$ and $P' = (x',y')$ be integral points of the fibration
with $x \neq x'$.  Then
\be
\label{eq:ip}
(P, P') = 2 - \deg\bigl(\gcd(x-x',y-y')\bigr).
\ee
\end{lemma}

Here $x$ and $x'$ (respectively $y$ and $y'$) are regarded as
sections of $\OO(4)$ (resp.~$\OO(6)$) on the projective \hbox{$t$-line};
when we represent them as polynomials in~$t$,
the degree of $\gcd(x-x',y-y')$ includes a term
$\min\bigl(4-\deg(x-x'), 6-\deg(y-y')\bigr)$.

{\em Proofs}: We give two proofs, one using the group law
and Lemma~\ref{lem:ht}, the other using the definition of $(P_i,P_j)$
as an intersection pairing.

{\em First proof.} Since
\be
\label{eq:ip_sum}
2 (P,P') = (P+P', P+P') - (P,P) - (P',P'),
\ee
and $(P,P) = (P',P') = 4$ because $P$ and $P'$ are assumed integral.
Thus it remains to find the canonical height of $P+P'$.  By the group law,
$P+P'$ has \hbox{$x$-coordinate}
\be
\label{eq:x(P+P')}
\left( \frac{y'-y}{x'-x} \right)^{\!2} - x - x'.
\ee
The denominator of (\ref{eq:x(P+P')}) is the square of the denominator of
$(y'-y) / (x'-x)$, and this denominator is $(x'-x) / \gcd(x'-x, y'-y)$.
Thus the degree of the denominator is
$2 \bigl[4 - \deg\bigl(\gcd(x'-x, y'-y)\bigr) \bigr]$.
Hence by Lemma~\ref{lem:ht} the height of $P+P'$ is
$12 - 2 \deg\bigl(\gcd(x'-x, y'-y)\bigr)$, so (\ref{eq:ip_sum}) gives
\be
\label{eq:ip_sum_ev}
2 (P,P') = 12 - 2 \deg\bigl(\gcd(x'-x, y'-y)\bigr) - 4 - 4
         = 4 - 2 \deg\bigl(\gcd(x'-x, y'-y)\bigr),
\ee
{}from which formula (\ref{eq:ip}) follows.

{\em Second proof.} Since our fibration has no irreducible fibers,
\be
\label{eq:ip_int}
(P,P') = - (s_P - s_0) \cdot (s_{P'} - s_0).
\ee
But $s_P \cdot s_0 = s_{P'} \cdot s_0 = 0$ because $P$ and $P'$ are integral,
and $s_0$ is a \hbox{$(-2)$-curve} so $-s_0 \cdot s_0 = 2$.  Hence
\be
\label{eq:ip_int1}
(P,P') = 2 - s_P \cdot s_{P'}.
\ee
But since $P$\/ and~$P'$ are integral, $s_P \cdot s_{P'}$ is just
the number of common roots of $x-x'$ and $y-y'$ with multiplicity
that is, the degree of $\gcd(x-x',y-y')$
(with $t=\infty$ accounted for as above). \qed

The \MoW\ lattice has $1311$ pairs of vectors of norm~$4$,
so we had numerous choices of basis.  We chose $P_1,\ldots,P_{17}$
so that each $x_i \in \Z[t]$ and $(P_1,P_i) = -2$ for $1 < i < 17$.
The latter condition allows for a shorter description of $P_i$
for $1 < i < 17$: instead of the \hbox{degree-$4$} polynomial $x_i$,
we need only list the ratio $m_i := (y_i-y_1)/(x_i-x_1)$,
which is a degree~$2$ polynomial by Lemma~\ref{lem:ip},
and turns out to have much smaller coefficients:
we compute that $m_2$ through $m_{16}$ are

{\tiny
\begin{verbatim}
{[3609606*t^2 + 17472654*t + 9306987, -13842426*t^2 - 47135036*t + 10995509, 5619654*t^2 + 34503044*t + 25228749,
   2903654*t^2 + 16309708*t + 25181877, -1674426*t^2 - 7417332*t + 1645077, -2879877/2*t^2 - 6644079/2*t + 35277,
  -1293435*t^2 - 74149065*t - 401684283, 598182*t^2 + 56959398*t - 102088725, 3845621*t^2 - 29297801*t - 58915563,
   3959553/2*t^2 + 21342168*t + 18889077, -2966262*t^2 + 2320578*t + 29039811, 2648379*t^2 + 907317*t - 59230479,
   1508406*t^2 + 20056077*t + 38513982, -1022424*t^2 + 32579784*t + 106182447, -4556214*t^2 - 12262536*t - 3300003]}
\end{verbatim}
}

For each $i = 2, \ldots, 16$ this determines $P_i$ up to the involution
$P \leftrightarrow -P_1 - P$, so $15$ further bits of information
determine these $15$ generators.
We could not make $(P_1,P_{17}) = -2$ as well,
because then $(P_1,P)$ would be $0 \bmod 2$ for all $P$ in the \MoW\ lattice,
and there is no such~$P_1$ of height only~$4$.  But we can use one of
the five $i$ such that $(P_i,P_{17}) = -2$; for example,
$(y_{17} - y_7) / (x_{17} - x_7)$ is
\verb:-183750*t^2 + 3253646*t - 11613683:\,.

A side effect of $(P_1,P_i) = -2$ for $1 < i < 17$ is that the Gram matrix
has no negative entries $P_i,P_j$ for $1 < i,j < 17$.  Indeed
$P_1 + P_i + P_j$ has height $4 + 2 (P_i, P_j)$, which implies
$(P_i, P_j) \geq 0$ because $P_1 + P_i + P_j \neq 0$
and every nonzero point has height at least~$4$.

\subsection{Fibers of rank at least $25$, $26$, $27$, and $28$.}
\label{ss:pair}

$\phantom\infty$

The fiber at $t = -9529 / 5471$ is the elliptic curve of rank at least~$28$
(and exactly $28$ under the Generalized Riemann Hypothesis (GRH)
for number fields \cite{KSW:GRH}) that we found on 1 May 2006
and announced two days later in \cite{NDE:NMBRTHRY28},
and which held the rank record until 2024.
This proves Theorem~\ref{thm:rk28}.

Our curves of ranks at least $25$, $26$, and $27$ \cite{Dujella} are
the fibers at $t = -2/377$, $t = -308/251$, and $t = 2456/135$ respectively.
The curve of rank $\geq 27$, also from 1 May 2006,
was first published in~\cite{KSW:GRH},
which proves that its rank is exactly~$27$ under GRH for number fields.
The curves of ranks at least~$25$ (found 21 April 2006)
and~$26$ (25 April 2006), were supplied on 7 January 2021
in reply to a query from Andrew Sutherland in order to fill
gaps in the then-known set of ranks of elliptic curves over~$\Q$
(the previous rank record was~$24$ \cite{rk24}).

\section{Quadratic base changes that increment the rank}
\label{s:18,19}

Let $f \in \NS (\XX)$ be the fiber class,
and suppose $C$ is a \hbox{$(-2)$-curve} on~$\XX$ such that $C \cdot f = 2$.
Then $t$ restricts to a \hbox{degree-$2$} function $C \to \PP^1_t$.
We call such~$C$\/ a ``rational quadratic section'' of the fibration.
Let $(\XX_C,C)$ be the base change of $(\XX,t)$ from the \hbox{$t$-line} to~$C$.
This elliptic fibration has the same \hbox{rank-$17$} group of sections
(composed with the quadratic cover $C \to \PP^1$,
and thus with all height pairings doubled),
and an extra section, call it $P_C$,
that takes every point $p \in C$ to $p$ itself.

\begin{lemma}
\label{em:rk18}
$P_C$ is independent of the pullback to $(\XX_C,C)$ of the
\MoW\ group of $(\XX,t)$.  Thus $(\XX_C,C)$ has \MoW\ rank at least~$18$.
\end{lemma}

{\em Proof}\/: The pullback of the \MoW\ group of $(\XX,t)$
is the subgroup of the \MoW\ group of $(\XX_C, C)$ fixed under
the Galois involution~$\iota$ of the quadratic cover~$C \to \PP^1_t$.
This involution takes $P_C$ to another section $P'_C$ with
$P_C\0 + P'_C$ Galois-invariant.  The difference $P_C\0 - P'_C$ is anti-invariant,
and intersects $s_0$ at the fixed points of~$\iota$ but not identically,
so no multiple can be in the invariant subgroup.\qed

The trace $\tau := P_C\0 + P'_C$ is some section of $(\XX,t)$.
Then the restriction of $s_\tau + s_0 - C$\/ to any fiber
is a principal divisor, and since $(\XX,t)$ has no reducible fibers
we must have $s_\tau + s_0 \sim C - nf$ for some~$n$.  Hence
\be
\label{eq:tau}
(s_\tau + s_0) \cdot (s_\tau + s_0) = (C - nf) \cdot (C - nf) = -2 - 4n.
\ee
The left-hand side is
\be
\label{eq:tau8}
s_\tau \cdot s_\tau + s_0 \cdot s_0 + 2 s_\tau \cdot s_0
= -2 - 2 + h(\tau) - 4 = h(\tau) - 8,
\ee
where $h(\tau)$ is the height of~$\tau$.
Thus $n = (6 - h(\tau)) / 4$, so $h(\tau) \equiv 2 \bmod 4$.
Conversely, suppose $\tau$ is in the \MoW\ group of $(\XX,t)$
with $h(\tau) \equiv 2 \bmod 4$, and let $n = (6 - h(\tau))/4 \in \Z$,
so $D_\tau := s_\tau + s_0 + nf$ is a divisor of self-intersection~$-2$.
Thus either the class $[D_\tau]$ or $[-D_\tau]$ contains an effective
divisor, and it must be $[D_\tau]$ because $D_\tau \cdot f = 2 > 0$.
Thus (again using the fact that there are no reducible fibers)
$D_\tau$ is linearly equivalent to either a rational quadratic section
or the sum of two sections.  The latter can certainly happen:
if $h(\tau) = 6$ then $n = 0$ and $D_\tau$ is already $s_0 + s_\tau$.
Translating $s_0 + s_\tau$ by any section~$P$, we find that more generally
$s_P + s_{P+\tau}$ is the effective divisor with trace $2P + \tau$.
In fact:

\begin{prop}
\label{prop:6}
The effective divisor in the class $[D_\tau]$ is a quadratic section
if and only there is no $P$ in the \MoW\ group such that $h(\tau - 2P) = 6$.
\end{prop}

{\em Proof}\/: We already saw that if $h(\tau - 2P) = 6$ then
$D_\tau \sim s_P + s_{P+\tau}$.
Conversely, if $D_\tau \sim s_P + s_{P'}$ then $\tau = P + P'$
in the group law, and $s_P \cdot s_{P'} = 1$ because
$D_\tau \cdot D_\tau = -2$, so $h(P'-P) = 6$ and $\tau = 2P + (P'-P)$.\qed

{\bf Remark.} \cite{Jeechul} extends this to a characterization of
rational quadratic sections of an elliptic fibration that is allowed to have
reducible fibers as long as all have type $A_n$ for some~$n$
(i.e.\ multiplicative of type $I_{n+1}$ or additive of type III or IV).

If a rational quadratic section~$C$\/ has trace~$\tau$,
then the translate of~$C$\/ by~$P$\/ has trace $C+2P$
and gives rise to a quadratic base change isomorphic with~$\XX_C$.
Thus we get a finite set of distinct quadratic base changes,
corresponding to the cosets mod~$2$ of the \MoW\ group of~$(\XX,t)$
that have no representatives of height~$6$.  We find these cosets
by using the GP function {\tt qfminim} to find all vectors of norm~$6$
in the \MoW\ lattice, finding that there are $39120$ cosets
that contain no such vectors.

To finish the proof of part~(a) of Theorem~\ref{thm:rk18,19},
we need only find one $C$\/ that has rational points.  In fact all our~$C$\/
have rational points, because we can always find a section $s_P$
such that $s_P \cdot C = 1$ (indeed any odd $s_P \cdot C$\/ would do).
Alternatively, it suffices to exhibit a rational point:
one of the simplest quadratic base changes is
\be
\label{eq:tau1}
\verb: u^2 = 4225*t^2 + 38636*t + 289444 :
\ee
which has a pair of rational points at infinity because
the leading coefficient $4225$ is $65^2$.

For part~(b), we choose $\tau,\tau'$
satisfying the condition of Proposition~\ref{prop:6},
and let $C,C'$ be the rational quadratic sections associated to $\tau,\tau'$.
Forming their compositum
yields a biquadratic base change to an elliptic fibration over
some \hbox{genus-$1$} curve $E_0$, with two new sections.
Again we use a Galois action, this time of $\{\pm1\}^2$,
to prove that we get a \MoW\ group of rank at least~$19$.
We then further require that $D_\tau \cdot D_{\tau'}$ be odd,
which happens for about half of the pairs $\tau,\tau'$.
Then $C \cap C'$ gives a divisor of odd degree on~$E_0$;
since $E_0$ already has divisors of degree~$4$ such as the preimage of $t=0$,
it follows that $E_0$ has a divisor of degree~$1$,
which is then effective by Riemann-Roch.
Hence $E_0$ has a rational point, which we use to make $E_0$ an elliptic curve.
We expect that this point together with its images under $\Gal(E_0/\PP^1_t)$
will generate a group of positive rank.  Again it is enough to check this
in a single example.  We combine (\ref{eq:tau1}) with another
quadratic base change
\be
\label{eq:tau2}
\verb: u^2 = 54756*t^2 - 3269604*t + 22473889$ :
\ee
which also has rational points at $t = \infty$
(leading coefficient $54756 = 234^2$).
We find that not only does the resulting point on the elliptic curve
\be
\label{eq:E0}
E_0: \verb: y^2 = x^3 + 1029367969*x^2 - 42900734074705920*x :
\ee
have infinite order, but $E_0$ has rank~$4$.
This is more than enough to prove~(b).
Part~(c) then follows by Silverman's specialization theorem.\qed

\section*{Acknowledgements}
Most of this work was done in 2006, and supported in part by
NSF grants DMS-0200687 and DMS-0501029 (the latter number is auspiciously
the smallest prime conductor of an elliptic curve of rank~$4$\ldots);
the only exception was the new choice of generators that lets us encode them
more efficiently.  Even that part, and the text, was done with no AI assistance
except that its writing was catalyzed by news of the AI-assisted work of
\cite{Rubinstein} and \cite[Curve~273]{ICARM}.
The computations were done in GP/PARI,
supplemented in a few spots with Magma function calls.

\end{document}